\documentclass[11pt]{amsart}
\usepackage{amsmath,amsfonts,amsthm,amssymb,graphicx,tikz,tikz-cd,url,hyperref}
\usepackage[all,2cell,ps]{xy}
\usepackage{thm-restate}
\usepackage[percent]{overpic}
\usepackage[T1]{fontenc}
\usepackage[utf8]{inputenc}

\theoremstyle{plain}
\newtheorem{thm}{Theorem}[section]

\newtheorem{lem}[thm]{Lemma}
\newtheorem{prop}[thm]{Proposition}

\newtheorem{cor}[thm]{Corollary}

\newtheorem{qtn}[thm]{Question}

\theoremstyle{definition}
\newtheorem{defn}[thm]{Definition}

\newtheorem{rem}[thm]{Remark}

\theoremstyle{remark}

\newcommand{\C}{\mathbb{C}}

\newcommand{\F}{\mathbb{F}}

\newcommand{\Hy}{\mathbb{H}}

\newcommand{\N}{\mathbb{N}}

\newcommand{\Q}{\mathbb{Q}}
\newcommand{\R}{\mathbb{R}}

\newcommand{\Z}{\mathbb{Z}}

\DeclareMathOperator{\SL}{SL}
\DeclareMathOperator{\PSL}{PSL}

\DeclareMathOperator{\tr}{tr}

\DeclareMathOperator{\Isom}{Isom}

\DeclareMathOperator{\Mod}{Mod}

\DeclareMathOperator{\area}{area}
\DeclareMathOperator{\leng}{leng}

\DeclareMathOperator{\geoc}{\overline{c}^*}
\DeclareMathOperator{\geoci}{\overline{c}^*_i}
\DeclareMathOperator{\barc}{\overline{c}}
\DeclareMathOperator{\barci}{\overline{c}_i}

\title{Finiteness of closed arithmetic hyperbolic surface bundles}
\author{Tam Cheetham-West}
\address{Department of Mathematics, Yale University, New Haven, CT, 06511}
\email{tamunonye.cheetham-west@yale.edu}
\author{Homin Lee}
\address{\hskip-\parindent
Center for Mathematical Challenges, Korea Institute for Advanced Study, Seoul, 02455 Republic of Korea}
\email{hominlee@kias.re.kr}
\author{Nicholas Miller}
\address{Department of Mathematics and Statistics, Villanova University, Villanova, PA, 19085}
\email{nicholas.miller@villanova.edu}
\date{}

\begin{document}

\begin{abstract}
For each $g\ge 2$, we show that there are finitely many cyclic commensurability classes of closed, fibered arithmetic hyperbolic $3$-manifolds with genus $g$ fiber.
More generally, we show that there are only finitely many conjugacy classes of admissible surface subgroups of $\PSL_2(\C)$ whose image is contained in the fundamental group of some arithmetic hyperbolic $3$-manifold.
We also give an effectively computable upper bound with effective asymptotic rate on both finiteness statements.
This affirms a conjecture of Bowditch, Maclachlan, and Reid.
\end{abstract}

\maketitle

\section{Introduction}
Surface bundles over the circle have played a pivotal role in the study of $3$-dimensional manifolds over the last several decades.
Indeed, seminal work of Thurston \cite{Thurston2} shows that, when the fiber surface has negative Euler characteristic, such bundles admit a hyperbolic metric of finite-volume precisely when the corresponding monodromy is pseudo-Anosov.
Moreover, Agol's resolution of Thurston's virtual fibering conjecture \cite{Agol,Thurston4} shows that every finite-volume hyperbolic $3$-manifold is finitely covered by such a surface bundle.
Consequently, every commensurability class of finite-volume hyperbolic manifolds contains infinitely many surface bundles.

In \cite{BMR}, Bowditch--Maclachlan--Reid initiated the study of arithmetic surface bundles. 
Recall that, despite Thurston and J\o rgensen's result that volumes of hyperbolic $3$-manifolds are well-ordered with order type $\omega^\omega$ \cite{Thurston} (see also \cite{NeumannZagier}), Borel \cite{Borel} proved that there are only finitely many arithmetic hyperbolic $3$-manifolds of volume less than any given number.
In particular, arithmetic structures are rare in the class of finite-volume hyperbolic $3$-manifolds.
See for instance \cite{BiringerSouto} for one instance of this.

Furthering this philosophy, Bowditch, Maclachlan, and Reid showed that there are finitely many cyclic commensurability classes of non-compact, arithmetic hyperbolic surface bundles with a fiber of fixed topological type \cite[Thm 4.2]{BMR}.
Here, two manifolds are cyclically commensurable if they share a finite-sheeted, cyclic cover, which is the natural equivalence relation for comparing fibered covers arising from powers of the monodromy.
As a sample consequence, for a surface of fixed non-compact topological type, $S_{g,n}$, there are only finitely many Bianchi groups whose commensurability class contains a surface bundle with fiber $S_{g,n}$. 
Bowditch, Maclachlan, and Reid go on to conjecture that the aforementioned finiteness should hold in general.
The purpose of this paper is to confirm that conjecture.

\begin{restatable}{thm}{mainthm}
\label{thm:finiteness}
Fix a natural number $g\ge 2$.
Then there are finitely many cyclic commensurability classes of closed, arithmetic hyperbolic surface bundles with fiber a closed surface of genus $g$.
\end{restatable}

\begin{figure}
    \centering
    \includegraphics[width=0.27\linewidth]{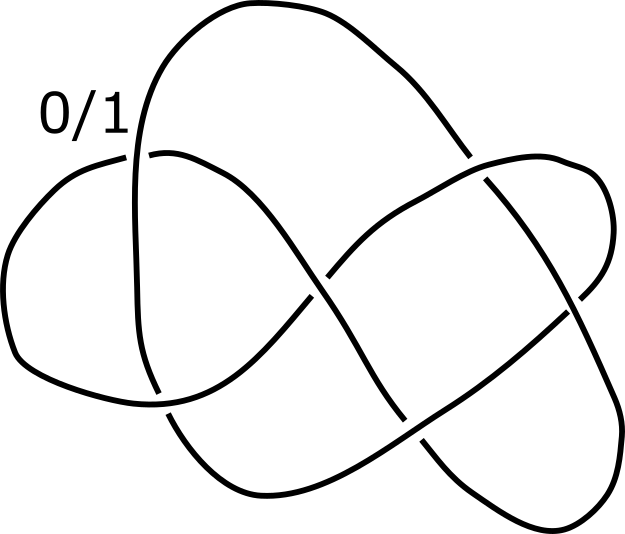}
    \includegraphics[width=0.27\linewidth]{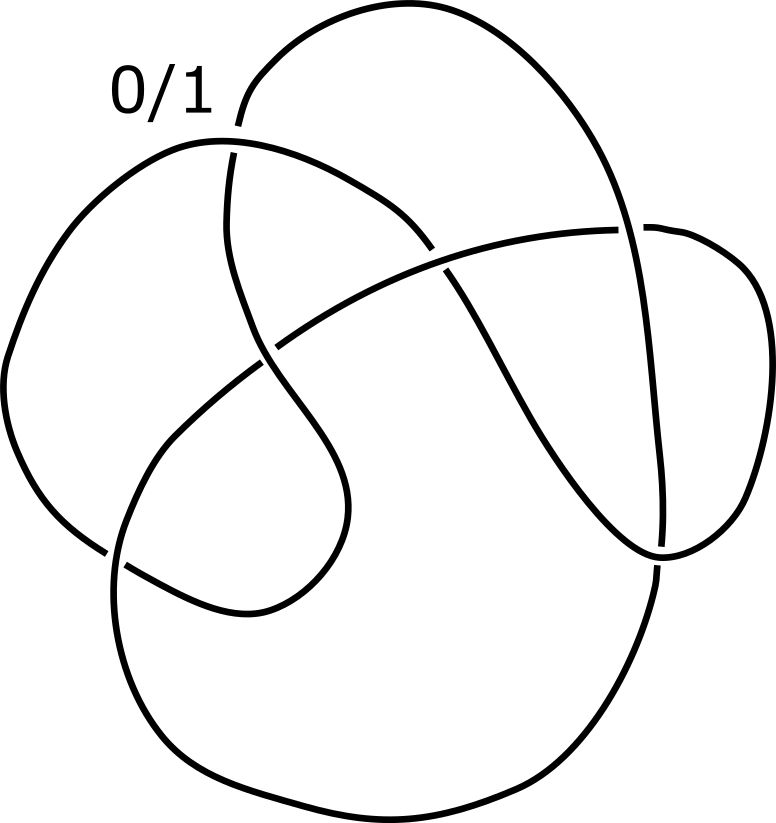}
    \includegraphics[width=0.27\linewidth]{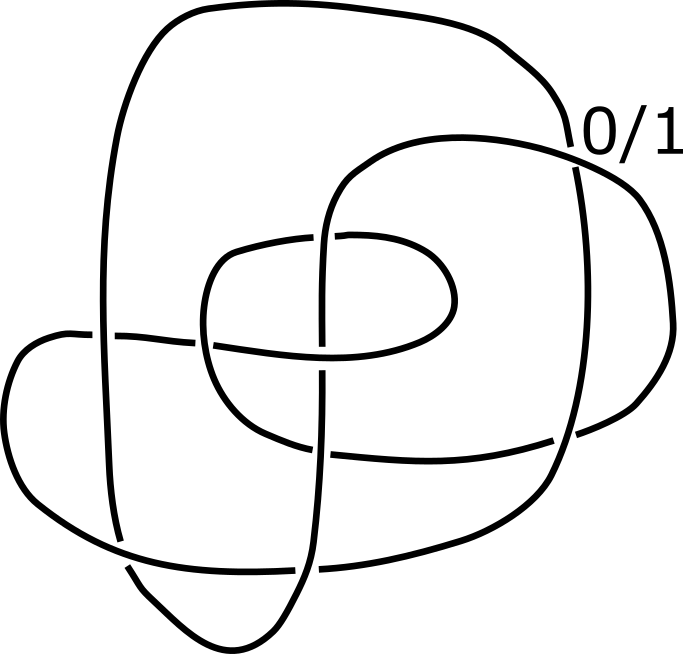}
    \caption{Surgery descriptions of some genus 2 closed arithmetic surface bundles with short monodromies in the mapping class group of the genus 2 surface. All examples in this figure are zero surgeries on genus 2 fibered knots and therefore fiber over the closed genus 2 surface \cite[Cor 8.19]{Gabai}.}
    \label{fig:placeholder}
\end{figure}

It was previously reasonable to expect that a typical mapping torus of a fixed surface is non-arithmetic. Indeed, Sisto--Taylor's theorem \cite[Thm 1.3]{SistoTaylor} that a random mapping torus has a short systole hints that the probability of observing a non-arithmetic mapping torus for a random mapping class $\omega_n$ in $\Mod(S_g)$ goes to $1$ as $n\to\infty$.
However, this is a probabilistic statement, as probability $1$ still does not imply a finiteness statement.
Therefore, Theorem \ref{thm:finiteness} gives the stronger deterministic statement and moreover gives an effective version of it (see Theorem \ref{thm:effective} below).

We now briefly outline the idea of the proof.
By \cite[Thm 4.3]{BMR} (see the discussion in Section \ref{sec:BMR}), it is enough to show that the invariant trace field of any such surface bundle is bounded in degree by a constant depending only on $g$.
The key new impetus for our work is the arithmetic Margulis lemma of Fr{\k{a}}czyk, Hurtado, and Raimbault \cite{FHR} (see the discussion in Section \ref{sec:AML}).
This result allows one to constrain, in an effective manner, a simplicial hyperbolic surface in the homotopy class of the fiber using the data of the trace field. 
As we will show in Section \ref{sec:proofsec}, as the degree of the invariant trace field increases, the arithmetic Margulis lemma forces the area of this simplicial hyperbolic surface to increase.
On the other hand, Gauss--Bonnet provides a bound on the area of such a surface depending only on $g$.
Hence, there must be a universal bound on the degree of the invariant trace field depending only on $g$.

As it turns out, though the main motivation of this paper is to prove Theorem \ref{thm:finiteness}, the ingredients of the proof do not depend on $M$ being fibered and we obtain the following generalization (see Section \ref{sec:generalization}).

\begin{thm}\label{thm:generalization}
Suppose that $g$ is a fixed natural number for which $g\ge 2$.
Then there are at most finitely many $\PSL_2(\C)$-conjugacy classes of admissible surface groups $\Delta=\pi_1(S_g)< \PSL_2(\C)$ such that $\Delta$ is contained in the fundamental group $\Gamma$ of some arithmetic hyperbolic $3$-manifold $M$.
\end{thm}

Here by an admissible surface group, we mean one that is discrete, faithful, and has no accidental parabolics.
The former two adjectives are automatic from the assumption on the image but the latter can occur when $\Gamma$ is not cocompact.
For non-compact, finite-volume hyperbolic surface subgroups, the analogous statement was proved by Bowditch, Maclachlan, and Reid \cite[Thm 4.3]{BMR} assuming that admissibility also included the condition of being type-preserving (that is, taking peripheral elements to parabolics). 
In particular, the confluence of these results shows that Theorem \ref{thm:generalization} holds for all finite-type hyperbolic surfaces.

As we discuss in Section \ref{sec:generalization}, it is impossible to conclude the stronger statement that there are finitely many commensurability classes of arithmetic hyperbolic $3$-manifolds whose image contains such $\Delta$.
However, as noticed in \cite[Cor 4.5]{BMR}, examining the proof of Theorem \ref{thm:generalization} one does obtain the following sharper corollary.

\begin{cor}\label{cor:generalization}
Suppose that $g$ is a fixed natural number for which $g\ge 2$.
Then there exist finitely many commensurability classes of arithmetic hyperbolic $3$-manifolds, $\mathcal{C}_1,\dots, \mathcal{C}_s$, such that for every $\PSL_2(\C)$-conjugacy class of admissible closed surface group $\Delta=\pi_1(S_g)<\PSL_2(\C)$ whose image lies in the fundamental group $\Gamma$ of an arithmetic hyperbolic $3$-manifold $M$ either
\begin{enumerate}
\item $M$ belongs to some $\mathcal{C}_i$, or
\item the inclusion $\Delta<\Gamma$ is induced by a totally geodesic immersion from an arithmetic hyperbolic surface.
\end{enumerate}
\end{cor} 

Note that in the latter case, we do not exclude the possibility that $M$ is also contained in some $\mathcal{C}_i$.
That is, it is entirely possible that the fundamental group of the same manifold $M$ contains two different $\PSL_2(\C)$-conjugacy classes of surface subgroups $\pi_1(S_g)$ in $\PSL_2(\C)$.
Again, the non-compact case was shown by Bowditch, Maclachlan, and Reid \cite[Cor 4.5]{BMR} and therefore Corollary \ref{cor:generalization} holds for all finite-type surfaces. 

Finally, we give an effective version of both Theorems \ref{thm:finiteness} and \ref{thm:generalization}.

\begin{thm}\label{thm:effective}
Suppose that $g$ is a fixed natural number for which $g\ge 2$.
Then there is an effectively computable function $f(g)$ so that there are at most finitely many $\PSL_2(\C)$-conjugacy classes of admissible surface groups $\Delta=\pi_1(S_g)< \PSL_2(\C)$ such that $\Delta$ is contained in the fundamental group $\Gamma$ of some arithmetic hyperbolic $3$-manifold $M$.
Moreover, $f(g)=e^{g^{3+o(1)}}$ with effectively computable constants.
In particular, the number of cyclic commensurability classes of closed, arithmetic surface bundles with a fiber of genus $g$ is asymptotically at most $e^{g^{3+o(1)}}$.
\end{thm}

We believe that the upper bound in the above theorem is far from optimal. 

\begin{qtn}
Is the number of cyclic commensurability classes of closed, arithmetic surface bundle with a fiber of genus $g$ is bounded above by a polynomial in $g$?
\end{qtn}

\bigskip\noindent \textbf{Acknowledgments:} The authors would like to thank Tarik Aougab, Junzhi Huang, and Yair Minsky for helpful conversations and give special thanks to Dave Futer for useful comments on the first draft of this paper.
The authors would also like to extend special thanks to the organizers of the Ventotene 2025 workshop where much of this work came together.
NM was partially supported by NSF grant DMS--2405264/2630020.
HL is supported in part by a KIAS Individual Grant (HP104101) via the June E Huh Center for Mathematical Challenges at Korea Institute for
Advanced Study, and Sanghyun Kim's Mid-Career Researcher Program (RS-2023-00278510) through the National Research Foundation funded by the government of Korea.

\bigskip\noindent \textbf{AI Disclosure:} 
The authors used ChatGPT 5.6 to proofread the paper at the final stage of the project for typos and notational consistency. AI was not used for any other purpose in this paper.

\section{Background}\label{sec:notation}

\subsection{Generalities on hyperbolic 3-manifolds}\label{sec:generalities}
Throughout this paper, we use $M$ to denote a finite-volume hyperbolic $3$-manifold and $\Gamma=\pi_1(M)$ to denote its fundamental group, unless otherwise specified.
By Mostow rigidity, $\Gamma$ can be realized as a subgroup of $\Isom^+(\Hy^3)\cong \PSL_2(\C)$ with matrix entries in a number field.
Under this embedding, given a subgroup $\Delta\le \Gamma$ we define the \emph{trace field} of $\Delta$ via the formula
$$K_\Delta=\Q(\{\tr(\widehat\delta)\mid \delta\in\Delta \}),$$
where $\widehat\delta$ is any lift of $\delta$ to $\SL_2(\C)$. This definition is independent of the choice of lift.
Similarly, we define the \emph{invariant trace field} of $\Delta$ as 
$$k\Delta=\Q(\{\tr(\widehat\delta)\mid \delta\in\Delta^{(2)} \}),$$
where $\Delta^{(2)}$ is the subgroup generated by squares.
It follows from Mostow rigidity and finite generation of $\Gamma$ that for any $\Delta\le \Gamma$, $K_\Delta$ and $k\Delta$ are always number fields. In fact, they are always subfields of the number fields $K_\Gamma$, $k\Gamma$ (respectively).

We say that two hyperbolic $3$-manifolds $M=\Hy^3/\Gamma$ and $M'=\Hy^3/\Gamma'$ are \emph{commensurable} if they share a finite sheeted cover.
The invariant trace field is a commensurability invariant of $M$, specifically, if $M$ and $M'$ are commensurable then $k\Gamma\cong k{\Gamma'}$ \cite[Thm 3.3.4]{MRBook}. 
Moreover, there is a further invariant called the invariant quaternion algebra, $A\Gamma$, which has the similar property \cite[Cor 3.3.5]{MRBook}. 

\begin{defn}
An arithmetic hyperbolic 3-manifold $M$ is a hyperbolic 3-manifold in the commensurability class of $\Hy^3/\text{P}(\mathcal{O}^1)$ for $\text{P}(\mathcal{O}^1)<\PSL_2(\C)$, where $\mathcal{O}^1$ is the group of norm-one elements of a maximal order $\mathcal{O}$ in a quaternion algebra $B/k$ over a number field $k$ with a single complex place such that $B\otimes_{k,\sigma} \R\cong \mathcal{H}$ (the Hamiltonian quaternions) at every real place $\sigma:k\hookrightarrow\R$. 
\end{defn}

In the setting that $M$ is arithmetic, the two invariants $k\Gamma$ and $A\Gamma$ are complete commensurability invariants, that is, $M$ and $M'$ are commensurable if and only if $k\Gamma\cong k{\Gamma'}$ and $A\Gamma\cong A{\Gamma'}$.
As it will not be relevant to this paper, we omit the definition of $A\Gamma$ and refer the reader to Maclachlan--Reid \cite{MRBook} for a thorough discussion of these invariants.

In the sequel, we are interested in closed, fibered hyperbolic $3$-manifolds.
Topologically, these arise as mapping tori of pseudo-Anosov maps $\phi$ of a genus $g$ surface $S_g$. 
We denote the corresponding manifold by $M_\phi$, which is a fiber bundle over the circle with fiber $S_g$ and carries a hyperbolic structure by work of Thurston \cite{Thurston2}. 
On the level of fundamental groups, one has the corresponding exact sequence
$$\xymatrix{1\ar[r]&\pi_1(S_g)\ar[r]&\Gamma\ar[r]&\Z\ar[r]&}1,$$
and hence the fundamental group of the fiber, $\Delta=\pi_1(S_g)$, is a normal subgroup of $\Gamma$.

When $M$ is arithmetic, the arithmeticity of $\Gamma$ is completely controlled by $\Delta$.
Indeed, the invariant trace field and invariant quaternion algebra are completely determined by $\Delta$, as $\Delta$ is a non-elementary, normal subgroup of $\Gamma$ \cite[Thm 4.3.1]{MRBook}. 
Moreover, by work of Bass \cite[Prop 2.8]{Bass}, normality also implies that if $\Delta$ has integral traces, then so does $\Gamma$.
These three properties characterize arithmeticity \cite[Thm 8.3.2]{MRBook} and therefore arithmeticity of $\Gamma$ is detected purely from the embedding of the fiber subgroup $\Delta$ in $\PSL_2(\C)$.

 \subsection{Simplicial hyperbolic surfaces}\label{sec:simplicialsurface}

Fix a one-vertex triangulation of $S_g$ with vertex $v_0$ and a distinguished edge $e_0$. Additionally, fix a closed geodesic $\alpha$ on $M_\phi$ with a distinguished point $x_0\in \alpha$. Then a \emph{useful simplicial surface adapted to $\alpha$} is a $\pi_1$-injective, $1$-Lipschitz map $\iota:S_g\to M$ for which $v_0\mapsto x_0$, $e_0\mapsto \alpha$, and such that every edge in the triangulation is mapped to a geodesic loop in $M$, every open face of the triangulation is a non-degenerate totally geodesic triangle, and the vertex $v_0$ has an interior angle of $\geq 2\pi$.
Pulling back the metric on $M$ via the map $\iota$ induces a singular hyperbolic path metric $\rho$ on $S_g$ with a single cone point of angle $\ge 2\pi$ at $v_0$.
In the sequel, we will frequently not reference the choice of curve $\alpha$ as it will not be needed in our argument.
We will also use bars to denote the image under the map $\iota$, e.g., $\overline{c}=\iota(c)$ for a curve $c$ on $S_g$.

In general, a useful simplicial surface need not be embedded.
Moreover, for every incompressible surface $f:S_g\hookrightarrow M$, there is a \emph{useful simplicial surface} in $M$ homotopic to $f$, as constructed by Canary in \cite{Canary} following \cite{Bonahon}. To see this surface, one can fix a one-vertex triangulation $T$ of $S_g$ containing $e_0\cup v_0$ and lift it to a triangulation $\widetilde{T}$ of $\widetilde{f}:\widetilde{S}_g\hookrightarrow \Hy^3$. A homotopy that takes $e_0\cup v_0$ to its unique closed geodesic representative $\geoc\subset  M$ will lift to a homotopy that sends the edges of the triangulation $\widetilde{T}$ to geodesic paths in $\Hy^3$ while the faces of the triangulation get sent to the lifts of geodesic triangles in $\Hy^3$. The surface $(S_g,T)$ so constructed is a \emph{simplicial prehyperbolic surface}. 

Canary's \emph{Non-Locally Strictly Convex} (NLSC) criterion \cite[Lem 4.2]{Canary} for vertices of a triangulable surface can be used to verify the angle $\geq 2\pi$ condition for vertices of a triangulation in a \emph{simplicial prehyperbolic surface} like $(S_g,T)$ constructed above.
Hence $(S_g,T)$ is, in fact, a \emph{useful simplicial surface}. A vertex $v$ of $(S_g,T)$ is said to be $NLSC$ if there is no open hemisphere in the unit tangent sphere $T^1_{\widetilde{f}(\widetilde{v})}(\Hy^3)$ of the vertex containing the link of $\widetilde{f}(\widetilde{v})$ in $\widetilde{f}(\widetilde{T})$. Since $T$ is a 1-vertex triangulation of $S_g$ and $e_0$ has been mapped to a closed geodesic $f(e_0)$ in $M$, it follows that the link of $\widetilde{f}(\widetilde{v})$ contains an antipodal pair corresponding to the two ends of $e_0$ in $T^1_{\widetilde{f}(\widetilde{v})}(\Hy^3)$, and thus the NLSC criterion is satisfied (see \cite[Lem 3.4]{Fan}). 

It is crucial for our purposes that $\iota:S_g\to M$ is 1-Lipschitz. This follows directly from the fact that away from the 1-skeleton of the triangulation, $\iota$ is totally geodesic in $M$. 
Indeed, straightening the edges makes every face an isometric immersion of a hyperbolic triangle, therefore, the length of any rectifiable path is not increased under the map $\iota$.
Moreover, it will be useful for us to normalize this surface by taking as small a cone angle as possible.
In particular, by spinning the one-vertex triangulation $T$ about the edge $e_0$, for any $\eta_0>0$ there is always a useful simplicial surface in the homotopy class of the fiber of a hyperbolic mapping torus for which the cone angle is in $[2\pi,2\pi+\eta_0]$ (see for instance \cite[Pg 239]{Thurston3}).

\subsection{Margulis tubes and the arithmetic Margulis lemma}\label{sec:AML}

Our argument will revolve around an analysis of $\delta$-Margulis tubes in the thin part of a closed hyperbolic $3$-manifold $M$.
We now collect the relevant details for the reader's convenience.

Fix $\delta>0$. 
Following the conventions of Futer--Purcell--Schleimer \cite{FPS}, we define the $\delta$-thin part of $M$ to be the set
$$M^{<\delta}=\{x\in M\mid \mathrm{injrad}_M(x)<\delta/2\},$$
where $\mathrm{injrad}_M(x)$ denotes the injectivity radius at $x$.
As is well known, this is precisely the set of points of $M$ for which there is a lift $\widetilde{x}$ of $x$ to $\Hy^3$ such that $d_{\Hy^3}(\widetilde{x},\gamma\widetilde{x})<\delta$ for some non-trivial $\gamma\in\Gamma$.
We will use $M^{\leq\delta}:=\overline{M^{<\delta}}$ for its closure in $M$.

Given a non-trivial $\gamma\in\Gamma$, we let $\geoc\!\!\!_\gamma$ denote the corresponding geodesic representative in $M$.
We also let
$$\mathcal{N}_\gamma=\Hy^3/\langle\gamma\rangle,$$
denote the corresponding nonsingular model solid torus, let $\geoc\!\!\!_\gamma^{\mathrm{mod}}$ denote its core geodesic.
We use $\mathcal N_\gamma^{<\delta}$ to denote its $\delta$-thin part and similarly set $\mathcal N_\gamma^{\leq\delta}=\overline{\mathcal N_\gamma^{<\delta}}$.
When this core geodesic is completely contained in $\mathcal{N}_\gamma^{\le\delta}$, there is some $R=R(\geoc\!\!\!_\gamma^{\mathrm{mod}},\delta)\ge 0$ for which the $R$-neighborhood, $N_R(\geoc\!\!\!_\gamma^{\mathrm{mod}})$, is completely contained in $\mathcal{N}_\gamma^{\le\delta}$.
We call the maximal such $R$ the \emph{model tube radius}.
In \cite[Prop 3.10]{FPS}, Futer, Purcell, and Schleimer give an effective estimate on the model tube radius.
In the non-singular case, which is the case we consider here, this formula can also be deduced from older work of Culler--Shalen \cite[\S1.3]{CS}.
Recall that the complex length of a loxodromic element $\gamma$ is the complex number whose real part is the translation length and imaginary part is the rotation angle.
\begin{prop}[Futer--Purcell--Schleimer]\label{prop:FPS}
For any $\delta>0$ such that $\geoc\!\!\!_\gamma^{\mathrm{mod}}\subset \mathcal{N}_\gamma^{\le \delta}$, the model tube radius is given by 
$$R=R(\geoc\!\!\!_\gamma^{\mathrm{mod}},\delta)=
\max_{\substack{n\in\N\\n\tau\le\delta}}
\left\{\cosh^{-1}\left(\sqrt{\frac{\cosh(\delta)-\cos(n\theta)}{\cosh(n\tau)-\cos(n\theta)}}\right)\right\},$$
where $\tau+i\theta$ is the complex length of $\gamma$.
\end{prop}
When $R$ is equal to the model tube radius, we will call the neighborhood $N_R(\geoc\!\!\!_\gamma^{\mathrm{mod}})$ the \emph{model $\delta$-Margulis tube associated to $\gamma$} and denote it by $\widehat{\mathbb{T}}_\delta(\gamma)$. 
Importantly, by construction, we always have the inclusion $\widehat{\mathbb{T}}_\delta(\gamma)\subseteq \mathcal{N}_\gamma^{\le\delta}$.

Work of Margulis (see \cite[\S12.6]{Ratcliffe}) shows that there is a universal $\mu_3$\footnote{Recall that we are confining ourselves to dimension $3$.
In general, the Margulis constant has a dependence on the dimension of the manifold.}, called the Margulis constant, such that every component of $M^{\le\delta}$ containing a closed geodesic is a solid torus, provided that $\delta<\mu_3$.
This can be translated into the algebraic statement that for any $\widetilde{x}\in\Hy^3$ the subgroup
$$\langle\{\gamma\in\Gamma\mid d_{\Hy^3}(\gamma \widetilde{x},\widetilde{x})<\delta\}\rangle\le\Gamma,$$
of $\Gamma$ is cyclic whenever $\delta<\mu_3$.
Recall the assumption that $\Gamma$ is cocompact and torsion-free.
However, recent work of Fr{\k{a}}czyk--Hurtado--Raimbault \cite[Thm 3.1]{FHR} building on work of Breuillard \cite{Breuillard} shows that a stronger statement holds when $\Gamma$ is arithmetic.
See \cite{CHL} for another proof of the main theorem of \cite{Breuillard}.
This is the content of the following, where only for this statement do we allow $\Gamma$ to be an arbitrary arithmetic lattice.

\begin{thm}[Fr{\k{a}}czyk--Hurtado--Raimbault]\label{thm:generalaml}
Let $G$ be a semi-simple Lie group. 
Then there exists an $\epsilon_G>0$ such that for every irreducible arithmetic lattice $\Gamma$ with adjoint trace field $k$ and every $\widetilde{x}\in\widetilde{M}\cong G/K$, the subgroup
\begin{equation}\label{eqn:subgroup}
\Gamma(\widetilde{x})=\left\langle \{\gamma\in\Gamma\mid d_{\widetilde{M}}(\widetilde{x},\gamma \widetilde{x})\le \epsilon_G[k:\Q]\}\right\rangle<\Gamma,\end{equation}
is virtually nilpotent.
\end{thm}
In the setting of $G=\PSL_2(\C)$ there is a recent strong arithmetic Margulis lemma due to Belolipetsky and Hurtado \cite{BH}, however, the right-hand side of their inequality depends on covolume; therefore, we opt to use this one for our purposes.

Returning to the case that $\Gamma$ is the fundamental group of a closed hyperbolic $3$-manifold, one deduces the following corollary, which we will refer to as the arithmetic Margulis lemma in the sequel.
\begin{cor}[Arithmetic Margulis lemma]\label{cor:aml}
Let $G=\PSL_2(\C)$.
There exists an $\epsilon_3=\epsilon_{\PSL_2(\C)}>0$ such that for any closed, arithmetic hyperbolic $3$-manifold $M$ with fundamental group $\Gamma=\pi_1(M)<\PSL_2(\C)$, the subgroup
$$\Gamma(\widetilde{x})=\left\langle \{\gamma\in\Gamma\mid d_{\Hy^3}(\widetilde{x},\gamma \widetilde{x})\le \delta\}\right\rangle<\Gamma,$$
is cyclic, where $\delta=\epsilon_3[k\Gamma:\Q]$ and $k\Gamma$ is the invariant trace field of $\Gamma$.
\end{cor}
Indeed, $\Gamma$ is torsion-free and cocompact, hence the only virtually nilpotent subgroups are cyclic.

The usual Margulis tube argument now shows that the model $\delta$-Margulis tube embeds in $M$ at this larger scale.
Indeed, let $\gamma\in\Gamma$ be primitive and suppose that $\geoc\!\!\!_\gamma\subset M^{\le\delta}$.
Let $\widetilde{\mathbb{T}}_\delta(\gamma)\subset\Hy^3$ denote the preimage of $\widehat{\mathbb{T}}_\delta(\gamma)$.
If
$$g\widetilde{\mathbb{T}}_\delta(\gamma)\cap\widetilde{\mathbb{T}}_\delta(\gamma)\neq\emptyset,$$
for some $g\in\Gamma$, then there are non-zero integers $m$, $n$ and a point $\widetilde{x}$ in this intersection such that
$$d_{\Hy^3}(\widetilde{x},\gamma^m\widetilde{x})\le\delta
,\qquad\qquad
d_{\Hy^3}(\widetilde{x},g\gamma^ng^{-1}\widetilde{x})\le\delta.$$
Corollary \ref{cor:aml} therefore implies that
$$\langle\gamma^m,g\gamma^ng^{-1}\rangle,$$
is cyclic.
As $\gamma$ is primitive, it follows that
$$g\langle\gamma\rangle g^{-1}=\langle\gamma\rangle,$$
and hence $g\in\langle\gamma\rangle$.
Thus the natural local isometry $\mathcal{N}_\gamma\longrightarrow M$ restricts to an embedding of $\widehat{\mathbb{T}}_\delta(\gamma)$.
This embedding identifies $\widehat{\mathbb{T}}_\delta(\gamma)$ with the $R$-neighborhood $N_R(\geoc\!\!\!_\gamma)$, where $R=R(\geoc\!\!\!_\gamma^{\mathrm{mod}},\delta)$ is the model tube radius from Proposition~\ref{prop:FPS}.
We call this image the \emph{$\delta$-Margulis tube around $\geoc\!\!\!_\gamma$} and denote it by $\mathbb{T}_\delta(\geoc\!\!\!_\gamma)$ and we call this $R$ the \emph{tube radius}.
Importantly, by construction, we always have the inclusion $\mathbb{T}_\delta(\geoc\!\!\!_\gamma)\subseteq M^{\le\delta}$.
Arithmeticity, therefore, gives one more refined control over Margulis tubes in the $\delta$-thin part of $M$.

In fact, there are other geometric properties that one also obtains finer control over in the presence of arithmeticity.
One such is the length of the systole, which is the shortest closed, essential geodesic on $M$.
Using a lower bound on the Mahler measure due to Dobrowolski \cite{Dobrowolski} (see also \cite[Thm 3.5]{BMR}), translated into a bound on the lengths of curves as in \cite[\S12.3]{MRBook}, one obtains a lower bound on the length of the systole and consequently the length of any closed geodesic.

To recall this, for every natural number $n\ge 2$ define
$$a(n)=\begin{cases}
\displaystyle
\ln\left(\frac{3+\sqrt{5}}{2}\right),&n=2,\\[6pt]
\ln\left(1+\frac{1}{1200}\left(\frac{\ln(\ln(n))}{\ln(n)}\right)^3\right),&n\ge 3.
\end{cases}$$
Then the log of the Mahler measure of any monic reciprocal, non-cyclotomic integer polynomial is bounded below by $a(n)$.
Indeed, the degree $2$ value in the definition of $a(n)$ is the elementary minimum for reciprocal, non-cyclotomic quadratic polynomials, while the values for $n\ge3$ come from Dobrowolski's bound.
The $n=2$ case is treated separately because the displayed Dobrowolski expression does not give a positive logarithmic lower bound when $n=2$.

As $\Gamma$ is arithmetic, \cite[Cor 8.3.5, 8.3.6]{MRBook} shows that for every non-trivial $\gamma\in\Gamma$, an eigenvalue of (any lift to $\SL_2(\C)$ of) $\gamma^2$ satisfies a reciprocal, non-cyclotomic integer polynomial of even degree $2m\le 2[K_\Gamma:\Q]$.
Consequently, we conclude that
$$\leng(c_\gamma)=\frac{1}{2}\leng(c_{\gamma^2})\ge \frac{1}{2}a(2m).$$
In particular, we obtain the following bound for arbitrary hyperbolic $3$-orbifolds.

\begin{thm}\label{thm:Mahlerbound}
For any arithmetic hyperbolic $3$-orbifold $M$ with trace field of degree $d$
$$\leng(c_\gamma)\ge\ell_0(d)
:=\frac{1}{2}\min\left\{a(2m)~\middle|~1\le m\le d,\ m\in\N\right\},$$
for any closed, essential geodesic $c_\gamma$ on $M$.
\end{thm}

\begin{rem}
Using the same work of Dobrowolski, one can more generally show that for any arithmetic lattice $\Gamma$ in a semisimple Lie group $G$ with adjoint trace field of degree $d$, the length of the geodesic $c_\gamma$ is
$$\leng(c_\gamma)\ge \frac{c_G}{\log(d+1)^3},$$
where $\gamma\in\Gamma$ is a non-compact, semisimple element and $c_G$ is an absolute constant depending only on $G$ (see \cite[Prop 2.6]{FHR} for instance).
This is a strictly weaker asymptotic lower bound than that of Theorem \ref{thm:Mahlerbound}, however, it would similarly be suitable for our applications.
\end{rem}

One can translate this into a lower bound on the length of all geodesics in a fibered hyperbolic $3$-manifold $M$ with invariant trace field of degree $d$ as follows, where we note that the key difference is the addition of the adjective invariant in front of trace field.

\begin{thm}\label{thm:Mahlerboundinvariant}
Given any closed, fibered arithmetic hyperbolic $3$-manifold $M=\Hy^3/\Gamma$ with genus $g$ fiber and invariant trace field of degree $d$, 
$$\leng(c_\gamma)\ge \ell_1(d)=\min\left\{\ell_0(k)~\middle|~d\le k\le 2^{2g+1}d,~k\in\N\right\},$$
for any closed, essential geodesic $c_\gamma$ on $M$.
\end{thm}

\begin{proof}
As $M$ is a mapping torus of $\phi:S_g\to S_g$ over the circle, it follows that
$$\dim_{\F_2}\left(H_1(M,\F_2)\right)\le 2g+1.$$
In particular, from \cite[Lem 3.3.3]{MRBook} we conclude that $[K_\Gamma:k\Gamma]\le 2^{2g+1}$.
The theorem then follows by noting that
$$d=[k\Gamma:\Q]\le [K_\Gamma:\Q]\le 2^{2g+1}d,$$
and applying Theorem \ref{thm:Mahlerbound}.
\end{proof}

It is worth mentioning that as $d\to\infty$, one checks that $\ell_1(d)\to 0$.
In particular, such a bound does not imply the famous short geodesic conjecture, which states that there should be a universal positive lower bound on the length of the systole of any arithmetic hyperbolic $3$-orbifold, independent of the trace field. 
This conjecture also implies Salem's conjecture in number theory, that the Mahler measures of Salem polynomials admit a non-trivial universal lower bound (see \cite[\S12.3]{MRBook}).

\subsection{Work of Bowditch--Maclachlan--Reid}\label{sec:BMR}

For the sake of completeness, in this section we recall the work of Bowditch, Maclachlan, and Reid \cite{BMR}, as it will be integral to our proof of Theorem \ref{thm:finiteness}.
We opt to only give sketches of the proofs of the statements we need and refer the reader to the original paper for the detailed arguments.

Recall from Section \ref{sec:generalities} that two hyperbolic $3$-manifolds are commensurable if they share a finite sheeted cover.
We further call them \emph{cyclically commensurable} if they share a common cyclic cover.
We remark that cyclic commensurability is genuinely a stronger condition than commensurability for the class of hyperbolic surface bundles (see \cite[\S4.4]{BMR} for a discussion of this).
With this in mind, Bowditch--Maclachlan--Reid prove the following theorem \cite[Cor 4.4]{BMR}, whose proof we now sketch.

\begin{thm}[Bowditch--Maclachlan--Reid]\label{thm:BMR}
Suppose that $d$, $g$ are fixed natural numbers and that $g\ge 2$.
Then there are at most finitely many cyclic commensurability classes of closed, fibered arithmetic hyperbolic $3$-manifolds with invariant trace field of degree at most $d$ and fiber a closed surface of genus $g$.
\end{thm}
\begin{proof}[Sketch of the proof]
Let $M=\Hy^3/\Gamma$, $M'=\Hy^3/\Gamma'$ be two closed, fibered hyperbolic $3$-manifolds with monodromies $\phi$, $\phi'$ (respectively). 
If $\Gamma$, $\Gamma'$ share a common fiber subgroup $\Delta\cong\pi_1(S_g)$, then we claim that $M$, $M'$ are cyclically commensurable.
Indeed, the ending laminations of the infinite cyclic cover $\Hy^3/\Delta$, which is homeomorphic to $S_g\times \R$, are fixed by both $\phi$, $\phi'$.
This implies that $\phi^r=(\phi')^s$ for some integers $r$, $s$ and therefore the fiber bundle with this monodromy is a finite cyclic cover of both $M$, $M'$.
As such, it suffices to show that, up to $\PSL_2(\C)$-conjugacy, there are at most finitely many subgroups $\Delta\cong\pi_1(S_g)< \PSL_2(\C)$ which occur as the fiber subgroup of an arithmetic fibered hyperbolic $3$-manifold group $\Gamma<\PSL_2(\C)$.

To this end, suppose that $\Delta<\Gamma<\PSL_2(\C)$ is as above.
Then, as described in Section \ref{sec:generalities} and the references therein, normality of $\Delta$ in $\Gamma$ implies that the invariant trace fields $k\Delta$, $k\Gamma$ of $\Delta$, $\Gamma$ coincide.
Moreover, as $H_1(S_g,\Z)\cong\Z^{2g}$, it follows from \cite[Lem 3.3.3]{MRBook} that if $K_\Delta$ is the (non-invariant) trace field of $\Delta$ then
$$[K_\Delta:\Q]=[K_\Delta:k\Delta][k\Delta:\Q]=[K_\Delta:k\Delta][k\Gamma:\Q]\le 2^{2g}d,$$
and hence the (non-invariant) trace field $K_\Delta$ is bounded above in degree as well.
Define the quantity
$$
\epsilon_0=\min\left\{\ell_1(k)~\middle|~2\le k\le d,\ k\in\N\right\},$$
which is necessarily positive.
By definition, the translation length of any non-trivial element of $\Delta$ is then at least $\epsilon_0$.
Therefore \cite[Thm 2.1]{BMR} shows that there exists some $V=V(\epsilon_0,g)$ and a set of generators $\{\delta_1,\dots,\delta_s\}$ of $\Delta$, whose size is bounded as a function of $g$, such that for all $i,j,k\in\{1,\dots,s\}$, we have
$$|\tr(\delta_i)|\le V,\qquad|\tr(\delta_i\delta_j)|\le V,\qquad|\tr(\delta_i\delta_j\delta_k)|\le V.$$
As $\Gamma$ is arithmetic, we additionally conclude that 
\begin{align*}|\sigma(\tr(\delta_i))|\le& \max\{2,V\},\qquad|\sigma(\tr(\delta_i\delta_j))|\le \max\{2,V\},\\
&|\sigma(\tr(\delta_i\delta_j\delta_k))|\le \max\{2,V\},\end{align*}
for any field embedding $\sigma:K_\Delta\hookrightarrow \C$.
Indeed, either $\sigma\vert_{k\Delta}$ is an embedding corresponding to the unique complex place of $k\Delta$ or $\sigma\vert_{k\Delta}$ is a real embedding for which the squares of such elements necessarily lie in Hamilton's quaternions.
Consequently, for this generating set, the set of all traces of generators, pairs of generators, and triples of generators is a collection of algebraic integers that are bounded in absolute value at every Archimedean place of $K_\Delta$.

We conclude by noting that the traces of generators, pairs of generators, and triples of generators determine the $\PSL_2(\C)$-conjugacy class of $\Delta$ uniquely.
Therefore, the finiteness of the set of all algebraic integers bounded at every Archimedean place in a number field of bounded degree implies the finiteness of such conjugacy classes.

On the other hand, the set of all such integers is the same as the set of monic, irreducible integer polynomials of bounded degree with all roots bounded above in absolute value by $\max\{2,V\}$. 
By rewriting the coefficients of such a polynomial in terms of elementary symmetric polynomials and applying Cauchy--Schwarz, one finds only finitely many possibilities for each coefficient, thereby concluding there are only finitely many such polynomials.
This completes the proof. 
\end{proof}

\begin{rem}\label{rem:BMR}
Though the above proof uses aspects specific to fibered hyperbolic $3$-manifolds, a more general statement is proved in \cite[Thm 4.3]{BMR}.
Precisely, Bowditch--Maclachlan--Reid show that for any $d\in\N$ and any topological surface $S_{g,n}$ of genus $g$ and $n$ punctures with negative Euler characteristic, there are finitely many $\PSL_2(\C)$-conjugacy classes of admissible $\pi_1(S_{g,n})$-representations whose image lies in an arithmetic Kleinian group with invariant trace field of degree bounded above by $d$.
\end{rem}


\section{The proof of Theorem \ref{thm:finiteness}}\label{sec:proofsec}

The goal of this section, and the strategy of proof for Theorem~\ref{thm:finiteness}, is to show that for a fixed genus $g\geq 2$ there is a natural number $d$ (depending only on $g$) such that every arithmetic surface bundle has invariant trace field of degree at most $d$. We accomplish this in several steps, which we sketch below.
As in the arithmetic Margulis lemma, we continue to use the quantity $\delta=\epsilon_3[k\Gamma:\Q]=\epsilon_3 d$ as well as the quantity $\eta_0$ from Section \ref{sec:simplicialsurface}.

First, in Proposition \ref{prop:bers} we show a Bers-type theorem -- that there is a pants decomposition of a useful simplicial hyperbolic surface for which the cuffs have uniformly bounded length. 
Assuming the degree of the invariant trace field is large enough, the $1$-Lipschitz property will imply that the images of these curves therefore lie in $\delta$-Margulis tubes in $M$.
In particular, collars around these pants curves will also map into the $\delta$-Margulis tubes of $M$.
We will quantify a lower bound on half-width of the collars for which this occurs in Propositions \ref{prop:tuberad}, \ref{prop:cylinderarea} and Corollaries \ref{cor:collar}, \ref{cor:cylinderarea}.
Finally, in the proof of Theorem \ref{thm:finiteness}, using the arithmetic Margulis lemma, we show that for the appropriate half-width (depending on $\delta$), these cylinders are embedded in the useful simplicial hyperbolic surface and have area that grows to infinity with $\delta$.
However, Gauss--Bonnet implies these areas cannot grow unboundedly and hence there must be a bound on $\delta$.


To this end, we now assume that, for a fixed natural number $g\ge 2$, $(S_g,\rho)$ is a useful simplicial hyperbolic surface inside of some closed hyperbolic surface bundle $(M,\rho_{\mathrm{hyp}})$, where $\rho_{\mathrm{hyp}}$ is the hyperbolic metric on $M$.
We begin with a minor generalization of the famous theorem of Bers \cite{Bers}.
This is likely well known to experts, however, the authors could not find an appropriate reference for it.

\begin{prop}[Bers decomposition]\label{prop:bers}
There is a constant $L_0=L_0(g)$ such that any useful simplicial hyperbolic surface $(S_g,\rho)$ has a pants decomposition $\{c_1,\dots,c_{3g-3}\}$, where each $c_i$ is a $\rho$-geodesic with $\leng_\rho(c_i)\le L_0$ for all $i$.
\end{prop}


\begin{proof}
For a fixed useful simplicial hyperbolic surface $(S_g,\rho)$, there is a Riemann surface $(S_g,\rho_0)$ in the same conformal class as $(S_g,\rho)$. Therefore, there is a pants decomposition $\{c'_1,\dots,c'_{3g-3}\}$ of $(S_g,\rho_0)$ with $\leng_{\rho_0}(c'_i)\le L_0$ for all $i$, where $L_0$ is the Bers constant for genus $g$ Riemann surfaces \cite{Bers}. By a theorem of Ahlfors \cite[Thm A]{Ahlfors},  $\leng_\rho(c'_i)\leq \leng_{\rho_0} (c'_i)\le L_0$ for all $i$. 
Letting $c_i$ be the corresponding $\rho$-geodesics in the homotopy class of $c_i'$ completes the proof.
\end{proof}


Fix $L_0$ as in Proposition \ref{prop:bers} and recall the notation from the end of the first paragraph in Section \ref{sec:simplicialsurface}.
For the remainder of the section, suppose also that $\delta>L_0$.
Then as useful simplicial hyperbolic surfaces are $1$-Lipschitz we have
$$\ell_1(\delta)\le \leng_M(\geoci)\le \leng_M(\barc_i)\le \leng_\rho(c_i)\le L_0<\delta,$$
for all $i$, where $\geoci$ denotes the geodesic representative of $\barci$ in $(M,\rho_{\mathrm{hyp}})$ and $\ell_1(\delta)$ is the lower bound on systole from Theorem \ref{thm:Mahlerboundinvariant}.
As $\delta=\epsilon_3 d$, any quantity that depends solely on $d$, can alternatively be written to depend only on $\delta$, which is why we write $\ell_1(\delta)$ instead of $\ell_1(d)$ in the equation above.

Using these inequalities, we obtain a quantitative bound on the distance between curves of $\rho_{\mathrm{hyp}}$-length at most $L_0$ and core $\rho_{\mathrm{hyp}}$-geodesics of $\delta$-Margulis tubes in $M$.
In what follows, given a subset $A\subset M$, $N_r(A)$ will denote the $r$-neighborhood of $A$ in $M$.
Moreover, we assume throughout the remainder of the section that $\delta$ is large enough so that $f(L_0,\delta)$ in the statement of the next proposition is well defined.


\begin{prop}\label{prop:tuberad}
Define the function
$$f(L_0,\delta)=\cosh^{-1}\left(\frac{\sinh(L_0/2)}{\sinh(\ell_1(\delta)/2)}\right).$$
Suppose that $\barc\subset M$ is a closed, primitive essential curve with geodesic representative $\geoc$. 
Suppose moreover that
$$\ell_1(\delta)\le \leng_M(\geoc)\le \leng_M(\barc)\le L_0.$$
Then $\barc\subset N_r(\geoc)$, where $r=r(\barc,\geoc)\le f(L_0,\delta)$.
\end{prop}


\begin{proof}
This follows from standard calculations in hyperbolic geometry.

Let $\widetilde{c}$ be a lift of $\geoc$ to $\Hy^3$ and let $\gamma\in\Gamma$ be a choice of primitive element whose axis is $\widetilde{c}$.
We may normalize by conjugating $\Gamma<\PSL_2(\C)$ so that if $$\Hy^3=\{(z,t)\mid z\in\C, t\in\R_{>0}\},$$
then $\widetilde{c}$ is identified with $\{0\}\times\R_{>0}\subset\Hy^3$ and
$$\gamma=\begin{pmatrix}
e^{\lambda/2}&0\\
0&e^{-\lambda/2}
\end{pmatrix},$$
where $\lambda=\tau+i\theta$ is the complex length of $\gamma$.
Given any $\widetilde{x}=(z_0,t_0)\in\Hy^3$, then $\gamma\widetilde{x}=(e^\lambda z_0,e^\tau t_0)$ and consequently, the usual distance formula gives that
\begin{align*}
\cosh(d_{\Hy^3}(\gamma\widetilde{x},\widetilde{x}))&=1+\frac{|e^\lambda z_0- z_0|^2+|e^\tau t_0-t_0|^2}{2e^\tau t_0^2},\\
&=\frac{|z_0|^2}{t_0^2}(\cosh(\tau)-\cos(\theta))+\cosh(\tau).
\end{align*}
Now suppose that $\widetilde{x}=(z_0,t_0)$ is a lift of a point on $\barc$ of maximal radial distance $r_{max}$ from $\geoc$, which also has distance $r_{\max}$ from $\widetilde{c}$ in the universal cover $\Hy^3$.
For any point $\widetilde{x}'=(0,t_0')\in \{0\}\times\R_{>0}$ it then follows that
$$\cosh(d_{\Hy^3}(\widetilde{x},\widetilde{x}'))=1+\frac{|z_0|^2+|t_0-t_0'|^2}{2t_0t_0'}.$$
This is minimized when $t_0'=\sqrt{|z_0|^2+t_0^2}$, that is, when
$$\cosh(d_{\Hy^3}(\widetilde{x},\widetilde{x}'))=1+\frac{|z_0|^2+|t_0-t_0'|^2}{2t_0t_0'}=\frac{\sqrt{|z_0|^2+t_0^2}}{t_0}.$$
Using that $\sinh^2(x)=\cosh^2(x)-1$, we find that
$$\sinh(r_{max})=\sinh(d_{\Hy^3}(\widetilde{x},\{0\}\times \R_{>0}))=\frac{|z_0|}{t_0}.$$
In particular, one deduces the inequality
\begin{align*}
\cosh(L_0)\ge \cosh(d_{\Hy^3}(\gamma\widetilde{x},\widetilde{x}))&= \sinh^2(r_{max})(\cosh(\tau)-\cos(\theta))+\cosh(\tau),\\
&\ge \sinh^2(r_{max})(\cosh(\tau)-1)+\cosh(\tau).
\end{align*}
Rearranging this, we obtain
$$\sinh^{2}(r_{max})\le\frac{\cosh(L_0)-\cosh(\tau)}{\cosh(\tau)-1}=\frac{\sinh^2(L_0/2)}{\sinh^2(\tau/2)}-1.$$
Using a final trig identity one concludes that
$$r_{max}\le\cosh^{-1}\left(\frac{\sinh(L_0/2)}{\sinh(\tau/2)}\right)\le \cosh^{-1}\left(\frac{\sinh(L_0/2)}{\sinh(\ell_1(\delta)/2)}\right).$$
Defining $f(L_0,\delta)$ to be the rightmost term concludes the proof.
\end{proof}


As the next corollary will quantify, for large enough $\delta$, it will therefore follow that $\barci$ is contained in $M^{<\delta}$ and hence in the $\delta$-Margulis tube around $\geoci$.
We additionally require a quantification on the distance of $\barci$ to the boundary of this tube.
In the following, we assume that $\delta$ is large enough so that $h(L_0,\delta)$ is well defined.


\begin{cor}\label{cor:collar}
Define
$$h(L_0,\delta)=\cosh^{-1}\left(\sqrt{\frac{\cosh(\delta)-1}{\cosh(L_0)+1}}\right)-\cosh^{-1}\left(\frac{\sinh(L_0/2)}{\sinh(\ell_1(\delta)/2)}\right),$$
then under the hypothesis of Proposition \ref{prop:tuberad}
$$d_M(\barc,\partial\mathbb{T}_\delta(\geoc))\ge h(L_0,\delta).$$
In particular, the $h(L_0,\delta)$-neighborhood around $c$ in $(S_g,\rho)$ has image contained in $\mathbb{T}_\delta(\geoc)$ under the $1$-Lipschitz map $\iota:S_g\to M$.
\end{cor}


\begin{proof}
Suppose that the complex length of $\geoc$ is $\tau+i\theta$. If $R=R(\geoc,\delta)$ denotes the tube radius of the $\delta$-Margulis tube in $M$, then 
$$d_M(\barc,\partial\mathbb{T}_\delta(\geoc))\ge R(\geoc,\delta)-r(\barc,\geoc)\ge R(\geoc,\delta)-f(L_0,\delta),$$
by Proposition \ref{prop:tuberad}.
Using the computation due to Futer--Purcell--Schleimer in Proposition \ref{prop:FPS} and the fact that the translation length $\tau\le L_0$, we conclude that
$$R(\geoc,\delta)\ge \cosh^{-1}\left(\sqrt{\frac{\cosh(\delta)-1}{\cosh(\tau)+1}}\right)\ge \cosh^{-1}\left(\sqrt{\frac{\cosh(\delta)-1}{\cosh(L_0)+1}}\right).$$
Consequently
$$d_M(\barc,\partial\mathbb{T}_\delta(\geoc))\ge \cosh^{-1}\left(\sqrt{\frac{\cosh(\delta)-1}{\cosh(L_0)+1}}\right)-\cosh^{-1}\left(\frac{\sinh(L_0/2)}{\sinh(\ell_1(\delta)/2)}\right).$$
where this final expression is precisely $h(L_0,\delta)$, by definition.

To conclude the proof, we note that this calculation shows that the entire $h(L_0,\delta)$-neighborhood of $\barc$ (in the metric $\rho_{\mathrm{hyp}}$) is contained in $\mathbb{T}_\delta(\geoc)$.
As the map $\iota:S_g\to M$ is $1$-Lipschitz, this implies that the $h(L_0,\delta)$-neighborhood of $c$ in $S_g$ has image contained in $\mathbb{T}_\delta(\geoc)$, as required.
\end{proof}


\begin{prop}\label{prop:cylinderarea}
Suppose that $c$ is the core $\rho$-geodesic of an embedded annulus, $A_i$, in $(S_g,\rho)$ with $\leng_\rho(c)\ge \ell$ and containing the metric tube of $\rho$--half-width at least $w$, then 
$$2\ell\sinh(w)-(\nu-2\pi)\le \area_\rho(A_i),$$
where $\nu$ is the cone angle at the singular point of $\rho$.
\end{prop}


\begin{proof}
For a regular value $0<r<w$, let $A(r)$ denote the closed metric $r$-neighborhood of $c$.
The boundary of $A(r)$ consists of two equidistant curves from $c$. For a metric of curvature at most $-1$ we have
$$\int_{\partial A(r)} k_g\,ds\ge 2\leng_\rho(c)\sinh(r)\ge2\ell\sinh(r),$$
where total geodesic curvature is understood to include any exterior angle contributions.

If the cone point belongs to $A(r)$, the cone Gauss--Bonnet formula gives
$$-\area_\rho(A(r))+\int_{\partial A(r)}k_g\,ds+(2\pi-\nu)=0.$$
If the cone point does not belong to $A(r)$, the same formula holds without the term $2\pi-\nu$.
In either case,
$$\area_\rho(A(r))\ge2\ell\sinh(r)-(\nu-2\pi).$$
Letting $r$ tend to $w$ through regular values and using $A(r)\subset A_i$ proves the result.
\end{proof}


Recalling our convention from Section \ref{sec:simplicialsurface} that we choose a useful simplicial hyperbolic surface with cone angle within $\eta_0$ of $2\pi$, we conclude the following.


\begin{cor}\label{cor:cylinderarea}
Let $\eta_0$ be as in Section \ref{sec:simplicialsurface}.
Suppose that $c$ is the core $\rho$-geodesic of an embedded annulus, $A_i$, in $(S_g,\rho)$ with $\leng_\rho(c)\ge \ell$ and containing the metric tube of $\rho$--half-width at least $w$, then 
$$2\ell\sinh(w)-\eta_0\le \area_\rho(A_i).$$
\end{cor}


\begin{proof}
This follows immediately from the fact that $\nu\in[2\pi,2\pi+\eta_0]$.    
\end{proof}


Combining the ingredients above, we now prove Theorem \ref{thm:finiteness}, which we restate for the reader's convenience. 


\mainthm*


\begin{proof}[Proof of Theorem \ref{thm:finiteness}]
Fix an arbitrarily small $\eta_0>0$ as in Section \ref{sec:simplicialsurface}.
Suppose that $M$ is a closed, arithmetic hyperbolic surface bundle and that $(S_g,\rho)$ is a useful simplicial hyperbolic surface in the homotopy class of the fiber with cone angle lying in $[2\pi,2\pi+\eta_0]$.
We may always find some such surface.
We continue the notation that $\iota:S_g\to M$ is the $1$-Lipschitz map and that $\barc$, $\barci$ denote the images of curves $c$, $c_i$ in $M$.
Moreover, we use bars over group elements to denote the image under the induced map of fundamental groups
$$\iota_*:\pi_1(S_g)\to \Gamma=\pi_1(M).$$
That is, $\overline{\gamma}=\iota_*(\gamma)\in \Gamma$.
By Theorem \ref{thm:BMR} of Bowditch--Maclachlan--Reid, it will suffice to show that there is an upper bound on $\delta$, depending only on $g$.

As we are only interested in an upper bound, we may assume that $\delta$ is large enough so that $\delta> L_0$, $f(L_0,\delta)$, $h(L_0,\delta)$ are well defined, and $h(L_0,\delta)$ is positive.
Define the quantity
$$w_{0}=w_{0}(\delta)=\min\left\{h(L_0,\delta),(\delta-L_0)/2\right\},$$
which is therefore positive by our assumptions.
Let $c_i\subset S_g$ be any pants curve in the fixed Bers pants decomposition of Proposition \ref{prop:bers} and choose a lift, $\widetilde{c}_i$, of $c_i$ to $\widetilde{S_g}$.
Let $\gamma_i\in\pi_1(S_g)$ be a choice of primitive element stabilizing the axis $\widetilde{c}_i$.
Define the $w_0$-tubular neighborhood of $\widetilde{c}_i$ in $\widetilde{S_g}$ by
$$\widetilde{N_i}=\widetilde{N}_{w_0}(\widetilde{c}_i)=\{\widetilde{x}\in \widetilde{S_g}\mid d_{\widetilde{\rho}}(\widetilde{x},\widetilde{c_i})<w_{0}\},$$
where $\widetilde{\rho}$ is a lift of $\rho$ to the universal cover.
Let $N_i=N_{w_0}(c_i)$ denote the image of $\widetilde{N_i}$ in $(S_g,\rho)$.
We claim that $N_i$ is an embedded annulus with core $\rho$-geodesic $c_i$.

Suppose $N_i$ were not embedded.
Then there would be some point $\widetilde{x}_0\in \widetilde{N_i}$ and an element $a\in\pi_1(S_g)\setminus\langle\gamma_i\rangle$ for which $d_{\widetilde{\rho}}(a\widetilde{x}_0,\widetilde{c}_i)<w_0$.
In particular, there is then some $m\in\Z$ such that
$$d_{\widetilde{\rho}}(\gamma_i^ma~\!\!\widetilde{x}_0,\widetilde{x}_0)<2w_{0}+\leng_\rho(c_i)\le 2w_0+L_0.$$
Lifting $\iota:S_g\to M$ to the universal cover, we obtain a map $\widetilde{\iota}:\widetilde{S_g}\to \Hy^3$ for which
\begin{equation}\label{eqn:translation}
d_{\Hy^3}(\widetilde{\iota}(\widetilde{x}_0),\overline{\gamma}_i^m\overline{a}\widetilde{\iota}(\widetilde{x}_0))< 2w_0+L_0\le \delta.
\end{equation}
Note that $\overline{N_i}=\iota(N_i)\subset \mathbb{T}_\delta(\geoci)$ by definition of $h(L_0,\delta)$, choice of $w_0$, and Corollary \ref{cor:collar}.
Moreover, there is some $k\in\N$ such that $\overline{\gamma}^k_i$ translates $\widetilde{\iota}(\widetilde{x}_0)$ by $\le \delta$ by definition of a $\delta$-Margulis tube.
Using the notation from Equation \eqref{eqn:subgroup}, this immediately contradicts the arithmetic Margulis lemma since
$$\langle \overline{\gamma}_i^m\overline{a},\overline{\gamma}^k_i\rangle\le\Gamma(\widetilde{\iota}(\widetilde{x}_0)),$$
implies that the latter subgroup is non-elementary.
Therefore no such $a$ exists and hence $\widetilde{N_i}$ projects to an embedded cylinder $N_i$ in $(S_g,\rho)$, which has half-width at least $w_0$ and length $\ell_i=\leng_\rho(c_i)$.

As $\ell_1(\delta)\le \ell_i$ for all $i$, we conclude by Gauss--Bonnet and Corollary \ref{cor:cylinderarea} that
$$ 2\ell_1(\delta)\sinh(w_0(\delta))-\eta_0\le\area_\rho(\cup_iN_i)\le \area(S_g,\rho)\le 2\pi(2g-2),$$
where we recall that $\eta_0$ is fixed independent of $\delta$.
Therefore
$$ 2\ell_1(\delta)\sinh(w_0(\delta))-\eta_0\le 4\pi(g-1).$$
Note that $w_0$ grows at least linearly in $\delta$ as $\delta\to\infty$ and consequently
\begin{equation}\label{eqn:genusbound}
D e^{C\delta}\frac{\ln(\ln(\delta))^3}{\ln(\delta)^3}\le 2\ell_1(\delta)\sinh(w_0(\delta))-\eta_0\le 4\pi(g-1),
\end{equation}
for some positive constants $C$, $D$ and $\delta$ sufficiently large.

However, the left-hand side goes to infinity with $\delta$ and therefore $\delta$ must be bounded by a constant depending only on $g$.
This implies $[k\Gamma:\Q]$ is bounded, completing the proof.
\end{proof}

\section{The proof of Theorem \ref{thm:generalization}}\label{sec:generalization}

We now describe the necessary modifications to the proof of Theorem \ref{thm:finiteness} in order to prove Theorem \ref{thm:generalization}.
A careful reading of the proof will reveal that fiberedness was only claimed to be used in exactly two places:
\begin{enumerate}
\item To find a useful simplicial surface in the homotopy class of the fiber of a fibered manifold, and
\item In the statement and sketch of the proof of Theorem \ref{thm:BMR} of Bowditch, Maclachlan, and Reid.
\end{enumerate}
If one suitably generalizes these two conditions, then the exact same proof as Theorem \ref{thm:finiteness} will immediately yield Theorem \ref{thm:generalization} and an examination of the proof will also yield Corollary \ref{cor:generalization} as a scholium.

First, we note that (1) holds in much greater generality than stated in Section \ref{sec:simplicialsurface}.
Indeed, results of Canary \cite[Section 5]{Canary} show that every immersed, essential surface of genus $g$ in a closed hyperbolic $3$-manifold $M$ is homotopic to a useful simplicial surface.
Moreover, as before, by spinning a triangulation one can assume that its cone angle is as close to $2\pi$ as one desires.
This gives the appropriate replacement for (1).

For (2), as mentioned in Remark \ref{rem:BMR}, Bowditch, Maclachlan, and Reid actually prove something slightly more general than Theorem \ref{thm:BMR}, which we now state only in the closed case.
Recall from the introduction that an admissible surface group is one which is discrete, faithful, and with no accidental parabolics.

\begin{thm}[Bowditch--Maclachlan--Reid]\label{thm:BMR2}
Suppose that $d$, $g$ are fixed natural numbers and that $g\ge 2$.
Then there are at most finitely many $\PSL_2(\C)$-conjugacy classes of admissible surface subgroups  $\Delta=\pi_1(S_g)<\PSL_2(\C)$ whose images lie in the fundamental group $\Gamma$ of an arithmetic hyperbolic $3$-manifold with invariant trace field of degree at most $d$.
\end{thm}

It is worth remarking that there is a subtle but important change to what is being shown to be finite in Theorem \ref{thm:BMR} as compared to Theorem \ref{thm:BMR2}.
Theorem \ref{thm:BMR2} would be false if one tried to analogously show that there are finitely many commensurability classes of arithmetic hyperbolic $3$-manifolds which contain an admissible surface subgroup of $\pi_1(S_g)$.
Indeed, one can show that closed, arithmetic hyperbolic surfaces totally geodesically immerse in infinitely many commensurability classes of arithmetic hyperbolic $3$-manifolds and hence that statement fails dramatically. 
However, when constructing such commensurability classes of $3$-manifolds, up to conjugation, one uses the composition of the embedding of the discrete faithful representation of $\Delta$ into $\PSL_2(\R)$ followed by the inclusion into $\PSL_2(\C)$ and therefore all such representations are $\PSL_2(\C)$-conjugate.
This justifies the formulation in Theorem \ref{thm:BMR2}.
The reason one can alternatively quantify over cyclic commensurability classes in Theorem \ref{thm:BMR} is precisely the extra structure afforded by requiring the image lies in the more restrictive class of fibered manifolds, and in particular by the use of ending laminations.

We now sketch the proof of Theorem \ref{thm:BMR2}.
To prove this, note that we only need a suitable replacement for the proof of Theorem \ref{thm:BMR} up to the definition of $\epsilon_0$, as the rest of the proof does not use the fibering structure.
To set notation, again assume that $\Delta<\Gamma<\PSL_2(\C)$, where $\Delta=\pi_1(S_g)$ is admissible and $\Gamma$ is the fundamental group of a finite-volume, arithmetic hyperbolic $3$-manifold $M$.

We first reduce to the case that $\Delta<\Gamma$ is not induced from a totally geodesic immersion of $S_g$ into $M$.
Suppose that it is and note that all totally geodesic submanifolds of an arithmetic manifold are themselves arithmetic.
In particular, up to $\PSL_2(\C)$-conjugation, $\Delta<\PSL_2(\C)$ is the composition of a discrete, faithful representation of the fundamental group of an arithmetic hyperbolic surface into $\PSL_2(\R)$ followed by the standard inclusion into $\PSL_2(\C)$.
However, there are only finitely many arithmetic surfaces of any genus $g$ \cite{Takeuchi}, which immediately implies finiteness of surface subgroups induced from totally geodesic immersions of surfaces $S_g$ into arithmetic hyperbolic $3$-manifolds.
The proof of Theorem \ref{thm:BMR2} is then finished by the following lemma, which gives the requisite bound for $\epsilon_0$.

\begin{lem}\label{lem:tglemma}
Suppose that $\Delta<\Gamma$ is not induced by a totally geodesic immersion of an arithmetic surface $S_g$ into $M$ and the degree of the invariant trace field $k\Gamma$ is at most $d$ over $\Q$, then the (non-invariant) trace field $K_\Delta$ has degree bounded above by $2^{2g}d$.
\end{lem}
\begin{proof}
As usual, let $k\Delta$, $k\Gamma$ denote the invariant trace fields of $\Delta$, $\Gamma$ respectively.
We first show that either the inclusion $\Delta<\Gamma$ is induced by a totally geodesic immersion of an arithmetic surface $S_g$ into $M$ or $k\Delta=k\Gamma$.
Indeed, arithmeticity of $M$ implies that $k\Gamma$ is a field with a unique complex place and therefore if $k\Delta\neq k\Gamma$ then $k\Delta$ is a proper subfield and consequently must be totally real.
As such, all traces of squares of elements of $\Delta$ are real which implies that $\Delta$ stabilizes a totally geodesic copy of $\Hy^2$ in $\Hy^3$ (see for instance \cite[Pg. 108]{Maskit}). 
That $\Delta$ is the fundamental group of an arithmetic surface then follows from a standard invariant quaternion algebra argument (see for instance \cite[Thm 9.5.2]{MRBook}).
Consequently, the hypotheses imply that $k\Delta=k\Gamma$.

As in the proof of Theorem \ref{thm:Mahlerboundinvariant}, one can bound the degree of the (non-invariant) trace field using the formula
$$[K_\Delta:k\Delta]\le 2^{\dim_{\F_2}(H_1(S_g;\F_2))}=2^{2g}.$$
From this the proof is concluded by noting that
$$[K_\Delta:\Q]=[K_\Delta:k\Delta][k\Delta:\Q]=[K_\Delta:k\Delta][k\Gamma:\Q]\le 2^{2g}[k\Gamma:\Q],$$
and using that $[k\Gamma:\Q]\le d$.
\end{proof}

We also record the length estimate needed to repeat the proof of Theorem \ref{thm:finiteness}.
If $\gamma\in\Delta$ is non-trivial, then $\tr(\widehat{\gamma}^{\,2})$ is an algebraic integer in $k\Delta=k\Gamma$.
Moreover, the argument of Theorem \ref{thm:Mahlerbound}, applied to $k\Gamma$, gives
\[
\leng(c_\gamma)\ge \frac{1}{2}a(2d)\ge \ell_0(d),
\]
for every non-trivial $\gamma\in\Delta$.
As before, since $\delta=\epsilon_3d$, we write $\ell_0(\delta)$ for $\ell_0(\delta/\epsilon_3)$.
We may therefore repeat the proof of Theorem \ref{thm:finiteness} to prove Theorem \ref{thm:generalization}, replacing every occurrence of $\ell_1(\delta)$ by $\ell_0(\delta)$.
For fixed $g$, the functions $\ell_1$ and $\ell_0$ have the same asymptotic behavior, so the final argument goes through unchanged.

We also mention that since $\Gamma$ is no longer required to be cocompact, virtual cyclicity should be replaced by virtual nilpotentency when applying the arithmetic Margulis lemma.
However, this does not present an issue since the subgroup in question contains a loxodromic element and, as $\Gamma$ is torsion-free, is therefore still cyclic.
Moreover, Lemma \ref{lem:tglemma} precisely gives Corollary \ref{cor:generalization}.

\section{Effectivization and the proof of Theorem \ref{thm:effective}}

In this section, we prove Theorem \ref{thm:effective} giving an explicitly computable asymptotic on the number of $\PSL_2(\C)$-conjugacy classes of admissible surface groups and hence on the number of cyclic commensurability classes of arithmetic surface bundles.

To do this, we first need a preliminary ingredient.
We mention that the constant $V_{\mathrm{BMR}}=V_{\mathrm{BMR}}(\epsilon_0,g)$ from the proof of Theorem \ref{thm:BMR} (which comes from \cite[Cor 2.2]{BMR}) may be taken as
$$V_{\mathrm{BMR}}=2\cosh(3K(g,0)L(g,0,\epsilon_0)),$$
where $K(g,0)$ is a certain combinatorial constant required to find a system of generators of ``standard type'' and $L(g,0,\epsilon_0)$ is the maximal diameter of a simplicial hyperbolic surface with a lower bound of $\epsilon_0$ on the injectivity radius.
However, Bowditch, Maclachlan, and Reid prove something stronger than is necessary for our purposes, as they allow precomposition by an automorphism and we are simply interested in the image subgroup.
Therefore as opposed to trying to understand the function $K(g,0)$ in their work, we instead prove the following which gives the $V$ needed in the proof of Theorem \ref{thm:BMR} with an effectively computable upper bound in $g$ and $\epsilon_0$.

\begin{lem}\label{lem:BMRconstant}
Suppose that $\Delta=\pi_1(S_g)<\PSL_2(\C)$ has the property that the translation length of every non-trivial element is at least $\epsilon_0$.
Then there exists some $V=V(\epsilon_0,g)$ and a set of generators $\{\delta_1,\dots,\delta_{2g}\}$ of $\Delta$, such that for all $i,j,k\in\{1,\dots,s\}$, we have
$$|\tr(\delta_i)|\le V,\qquad|\tr(\delta_i\delta_j)|\le V,\qquad|\tr(\delta_i\delta_j\delta_k)|\le V.$$
Moreover, $V$ is bounded above by
$$V=2\cosh\left(\frac{6\epsilon_0(g-1)}{\cosh(\epsilon_0/2)-1}\right).$$
In particular, when $0<\epsilon_0\le 1$, it follows that $V=V(g,\epsilon_0)=e^{O(g/\epsilon_0)}$.
\end{lem}
\begin{proof}
As above, let $L(g,0,\epsilon_0)$ be the diameter of the simplicial hyperbolic surface with injectivity radius at least $\epsilon_0$.
That 
$$L(g,0,\epsilon_0)\le \frac{2\epsilon_0(g-1)}{\cosh(\epsilon_0/2)-1},$$
follows from a standard packing argument (explicitly written in \cite[Prop. 2.7]{BMR}) using embedded balls of radius $\epsilon_0/2$ on a hyperbolic surface combined with the Gauss--Bonnet theorem.
Moreover, \cite[Lem 2.5]{BMR} shows that there is a system of generators $\{\delta_1,\dots,\delta_{2g}\}$ of $\Delta$ with length at most $2L(g,0,\epsilon_0)$ and hence all of the trace quantities above are bounded by $3$ times this.
Using the relationship between traces and lengths, it therefore follows that 
$$V\le 2\cosh\left(\frac{6L(g,0,\epsilon_0)}{2}\right)=2\cosh\left(\frac{6\epsilon_0(g-1)}{\cosh(\epsilon_0/2)-1}\right),$$
as required.
The remaining statement follows from the fact that
$$\cosh(\epsilon_0/2)-1\ge \epsilon_0^2/8,$$
in the aforementioned regime. 
\end{proof}

We now prove Theorem \ref{thm:effective}.

\begin{proof}[Proof of Theorem \ref{thm:effective}]
That everything in this paper can be bounded by an effectively computable function of $g$ is immediate.
Indeed, $\epsilon_0$ is clearly effectively computable, $\epsilon_3$ is effectively computable by the proof of \cite[Thm~3.1]{FHR}, $w_0$ is effectively computable as a function of $d$, a bound on $d$ in terms of $g$ is effectively computable by Equation \eqref{eqn:genusbound}, the function $V(\epsilon_0,g)$ is effectively computable by Lemma \ref{lem:BMRconstant}, and hence the number of polynomials considered in the proof of Theorem \ref{thm:BMR} is effectively computable as a function of $g$.

What remains is to show the asymptotic rate.
We will put the ingredients together using standard asymptotic notation.
Note first that by the bound of Dobrowolski, we have
$$\epsilon_0=\Omega\left(\left(\frac{\ln(\ln(d))}{\ln(d)}\right)^3\right),$$
and consequently Lemma \ref{lem:BMRconstant} shows that
$$V=e^{O\left(g\left(\frac{\ln(d)}{\ln(\ln(d))}\right)^3\right)}.$$
Since we have already shown that $d=O(g)$ in Equation \eqref{eqn:genusbound}, we therefore obtain that
$$V=e^{O\left(g\left(\frac{\ln(g)}{\ln(\ln(g))}\right)^3\right)}.$$

To obtain the best upper bound possible, we now must try to optimize the remaining step of the proof of Theorem \ref{thm:BMR}, which is to efficiently count the number of traces needed to specify a surface representation as well as to efficiently count the number of integer polynomials possible.
A priori the argument in the proof of Theorem \ref{thm:BMR} shows that the number of traces of generators needed is $O(g^3)$, however this can be improved.
Indeed the paragraph following \cite[Prop 2.3]{BMR} shows that not all traces of products of pairs and triples of generators are necessary to fix the conjugacy class of such a representation. As written therein, one only needs the traces of elements from the set
$$\{\delta_i\}_{i=1}^{2g}\cup\{\delta_1\delta_i\}_{i=2}^{2g}\cup\{\delta_2\delta_i\}_{i=3}^{2g}\cup\{\delta_1\delta_2\delta_i\}_{i=3}^{2g},$$
which one readily sees has cardinality $8g-5$. 
Therefore we can improve this asymptotic from $O(g^3)$ to $O(g)$.

For the count of polynomials, rather than count the traces themselves, we count their squares in the invariant trace field $k\Gamma$.
Indeed, noting that
$$\tr(A^2)=\tr(A)^2-2,$$
for each of the aforementioned elements $A$, the algebraic integer $\tr(A^2)$ lies in $k\Gamma$ and has degree at most $d=[k\Gamma:\Q]=O(g)$.
Moreover, all but at most two of its conjugates have absolute value at most $2$, while the remaining two have absolute value at most $V^2+2$.

Let $p(x)=x^m+c_1x^{m-1}+\cdots+c_m$ be the minimal polynomial of one such squared trace, where $m\le d$.
Using elementary symmetric polynomials, each $c_i$ is the sum of at most $2^m$ products of roots and each such product has absolute value at most $2^m(V^2+2)^2$.
Therefore
$$|c_i|\le 4^m(V^2+2)^2\le 4^d(V^2+2)^2.$$
Since $d=O(g)$ it follows that
$$|c_i|\le e^{O\left(g\left(\frac{\ln(g)}{\ln(\ln(g))}\right)^3\right)}.$$
which gives the asymptotic for the number of possibilities of each $c_i$.
Since the degree is at most $d=O(g)$, there are therefore at most
$$e^{O\left(g^2\left(\frac{\ln(g)}{\ln(\ln(g))}\right)^3\right)},$$
possible minimal polynomials, and hence at most this many possible squared traces, for each of the $8g-5$ elements listed above.

Finally, each squared trace determines the corresponding ordinary trace up to at most two choices.
It follows that the total number of possible trace tuples is bounded above by
$$\left(2e^{O\left(g^2\left(\frac{\ln(g)}{\ln(\ln(g))}\right)^3\right)}
\right)^{8g-5}=e^{O\left(g^3\left(\frac{\ln(g)}{\ln(\ln(g))}\right)^3\right)}.$$
We therefore conclude that the number of such conjugacy classes is bounded above by the asymptotic
$$e^{O\left(g^3\left(\frac{\ln(g)}{\ln(\ln(g))}\right)^3\right)}=e^{g^{3+o(1)}},$$
with effectively computable constants, as required.
\end{proof}


\bibliographystyle{abbrv}
\bibliography{biblio}

@article{Gabai,
  author  = {Gabai, David},
  title   = {Foliations and the topology of 3-manifolds. {III}},
  journal = {J. Differential Geom.},
  volume  = {26},
  number  = {3},
  year    = {1987},
  pages   = {479--536},
  doi     = {10.4310/jdg/1214441488},
  url     = {https://doi.org/10.4310/jdg/1214441488}
}

@article{SistoTaylor,
  author  = {Sisto, Alessandro and Taylor, Samuel J.},
  title   = {Largest projections for random walks and shortest curves in random mapping tori},
  journal = {Math. Res. Lett.},
  volume  = {26},
  number  = {1},
  year    = {2019},
  pages   = {293--321},
  doi     = {10.4310/MRL.2019.v26.n1.a14},
  url     = {https://doi.org/10.4310/MRL.2019.v26.n1.a14}
}

@article{BiringerSouto,
  author  = {Biringer, Ian and Souto, Juan},
  title   = {A finiteness theorem for hyperbolic 3-manifolds},
  journal = {J. Lond. Math. Soc. (2)},
  volume  = {84},
  number  = {1},
  year    = {2011},
  pages   = {227--242},
  doi     = {10.1112/jlms/jdq106},
  url     = {https://doi.org/10.1112/jlms/jdq106}
}

@misc{Fan,
  author        = {Fan, Carol E.},
  title         = {Injectivity radius bounds in hyperbolic {$I$}-bundle convex cores},
  year          = {1999},
  eprint        = {math/9907052},
  archiveprefix = {arXiv},
  primaryclass  = {math.GT},
  url           = {https://arxiv.org/abs/math/9907052}
}

@article{Bonahon,
  author  = {Bonahon, Francis},
  title   = {Bouts des vari{\'e}t{\'e}s hyperboliques de dimension 3},
  journal = {Ann. of Math. (2)},
  volume  = {124},
  number  = {1},
  year    = {1986},
  pages   = {71--158},
  doi     = {10.2307/1971388},
  url     = {https://doi.org/10.2307/1971388}
}

@article{Canary,
  author  = {Canary, Richard D.},
  title   = {A covering theorem for hyperbolic 3-manifolds and its applications},
  journal = {Topology},
  volume  = {35},
  number  = {3},
  year    = {1996},
  pages   = {751--778},
  doi     = {10.1016/0040-9383(94)00055-7},
  url     = {https://doi.org/10.1016/0040-9383(94)00055-7}
}

@incollection{Bers,
  author    = {Bers, Lipman},
  title     = {An inequality for {R}iemann surfaces},
  booktitle = {Differential Geometry and Complex Analysis},
  editor    = {Chavel, Isaac and Farkas, Hershel M.},
  publisher = {Springer},
  address   = {Berlin},
  year      = {1985},
  pages     = {87--93},
  doi       = {10.1007/978-3-642-69828-6_7},
  url       = {https://doi.org/10.1007/978-3-642-69828-6_7}
}

@incollection{Thurston2,
  author    = {Thurston, William P.},
  title     = {Hyperbolic structures on 3-manifolds, {II}: Surface groups and 3-manifolds which fiber over the circle},
  booktitle = {Collected Works of William P. Thurston with Commentary. Vol. II: 3-Manifolds, Complexity and Geometric Group Theory},
  note      = {August 1986 preprint; January 1998 eprint},
  publisher = {American Mathematical Society},
  address   = {Providence, RI},
  year      = {2022},
  pages     = {79--110}
}

@article{Ahlfors,
  author  = {Ahlfors, Lars V.},
  title   = {An extension of {S}chwarz's lemma},
  journal = {Trans. Amer. Math. Soc.},
  volume  = {43},
  number  = {3},
  year    = {1938},
  pages   = {359--364},
  doi     = {10.1090/S0002-9947-1938-1501949-6},
  url     = {https://doi.org/10.1090/S0002-9947-1938-1501949-6}
}

@article{Bass,
  author  = {Bass, Hyman},
  title   = {Groups of integral representation type},
  journal = {Pacific J. Math.},
  volume  = {86},
  number  = {1},
  year    = {1980},
  pages   = {15--51},
  url     = {https://projecteuclid.org/journals/pacific-journal-of-mathematics/volume-86/issue-1/Groups-of-integral-representation-type/pjm/1102780613.full}
}

@article{BMR,
  author  = {Bowditch, B. H. and Maclachlan, C. and Reid, A. W.},
  title   = {Arithmetic hyperbolic surface bundles},
  journal = {Math. Ann.},
  volume  = {302},
  number  = {1},
  year    = {1995},
  pages   = {31--60},
  doi     = {10.1007/BF01444486},
  url     = {https://doi.org/10.1007/BF01444486}
}

@article{NeumannZagier,
  author  = {Neumann, Walter D. and Zagier, Don},
  title   = {Volumes of hyperbolic three-manifolds},
  journal = {Topology},
  volume  = {24},
  number  = {3},
  year    = {1985},
  pages   = {307--332},
  doi     = {10.1016/0040-9383(85)90004-7},
  url     = {https://doi.org/10.1016/0040-9383(85)90004-7}
}

@article{Borel,
  author  = {Borel, Armand},
  title   = {Commensurability classes and volumes of hyperbolic 3-manifolds},
  journal = {Ann. Scuola Norm. Sup. Pisa Cl. Sci. (4)},
  volume  = {8},
  number  = {1},
  year    = {1981},
  pages   = {1--33},
  url     = {https://www.numdam.org/item/ASNSP_1981_4_8_1_1_0/}
}

@article{Thurston4,
  author  = {Thurston, William P.},
  title   = {Three-dimensional manifolds, {K}leinian groups and hyperbolic geometry},
  journal = {Bull. Amer. Math. Soc. (N.S.)},
  volume  = {6},
  number  = {3},
  year    = {1982},
  pages   = {357--381},
  doi     = {10.1090/S0273-0979-1982-15003-0},
  url     = {https://doi.org/10.1090/S0273-0979-1982-15003-0}
}

@article{Agol,
  author  = {Agol, Ian},
  title   = {The virtual {H}aken conjecture},
  note    = {With an appendix by Ian Agol, Daniel Groves, and Jason Manning},
  journal = {Doc. Math.},
  volume  = {18},
  year    = {2013},
  pages   = {1045--1087},
  url     = {https://www.elibm.org/article/10000267}
}

@article{Thurston3,
  author  = {Thurston, William P.},
  title   = {Hyperbolic structures on 3-manifolds. {I}. Deformation of acylindrical manifolds},
  journal = {Ann. of Math. (2)},
  volume  = {124},
  number  = {2},
  year    = {1986},
  pages   = {203--246},
  doi     = {10.2307/1971277},
  url     = {https://doi.org/10.2307/1971277}
}

@book{Ratcliffe,
  author    = {Ratcliffe, John G.},
  title     = {Foundations of Hyperbolic Manifolds},
  series    = {Graduate Texts in Mathematics},
  volume    = {149},
  edition   = {Second},
  publisher = {Springer},
  address   = {New York},
  year      = {2006},
  doi       = {10.1007/978-0-387-47322-2},
  url       = {https://doi.org/10.1007/978-0-387-47322-2}
}

@article{CHL,
  author  = {Chen, Lvzhou and Hurtado, Sebastian and Lee, Homin},
  title   = {A height gap in {$\mathrm{GL}_d(\overline{\mathbb Q})$} and almost laws},
  journal = {Groups Geom. Dyn.},
  volume  = {19},
  number  = {3},
  year    = {2025},
  pages   = {899--912},
  doi     = {10.4171/GGD/800},
  url     = {https://doi.org/10.4171/GGD/800}
}

@article{Breuillard,
  author  = {Breuillard, Emmanuel},
  title   = {A height gap theorem for finite subsets of {$\mathrm{GL}_d(\overline{\mathbb Q})$} and nonamenable subgroups},
  journal = {Ann. of Math. (2)},
  volume  = {174},
  number  = {2},
  year    = {2011},
  pages   = {1057--1110},
  doi     = {10.4007/annals.2011.174.2.7},
  url     = {https://doi.org/10.4007/annals.2011.174.2.7}
}

@article{Dobrowolski,
  author  = {Dobrowolski, Edward},
  title   = {On a question of {L}ehmer and the number of irreducible factors of a polynomial},
  journal = {Acta Arith.},
  volume  = {34},
  number  = {4},
  year    = {1979},
  pages   = {391--401},
  doi     = {10.4064/aa-34-4-391-401},
  url     = {https://doi.org/10.4064/aa-34-4-391-401}
}

@book{MRBook,
  author    = {Maclachlan, Colin and Reid, Alan W.},
  title     = {The Arithmetic of Hyperbolic 3-Manifolds},
  series    = {Graduate Texts in Mathematics},
  volume    = {219},
  publisher = {Springer-Verlag},
  address   = {New York},
  year      = {2003},
  doi       = {10.1007/978-1-4757-6720-9},
  url       = {https://doi.org/10.1007/978-1-4757-6720-9}
}

@article{FPS,
  author  = {Futer, David and Purcell, Jessica S. and Schleimer, Saul},
  title   = {Effective distance between nested {M}argulis tubes},
  journal = {Trans. Amer. Math. Soc.},
  volume  = {372},
  number  = {6},
  year    = {2019},
  pages   = {4211--4237},
  doi     = {10.1090/tran/7678},
  url     = {https://doi.org/10.1090/tran/7678}
}

@book{Thurston,
  author    = {Thurston, William P.},
  title     = {The Geometry and Topology of Three-Manifolds. Vol. {IV}},
  note      = {Edited and with a preface by Steven P. Kerckhoff and a chapter by J. W. Milnor},
  publisher = {American Mathematical Society},
  address   = {Providence, RI},
  year      = {2022}
}

@article{BH,
  author        = {Belolipetsky, Mikhail and Hurtado, Sebastian},
  title         = {A strong height gap theorem for ${PGL}_2$},
  journal       = {arXiv preprint arXiv:2507.22266},
  year          = {2025},
  eprint        = {2507.22266},
  archivePrefix = {arXiv},
  primaryClass  = {math.GR},
}

@article{FHR,
  author  = {Fr{\k{a}}czyk, Miko{\l}aj and Hurtado, Sebastian and Raimbault, Jean},
  title   = {Topological complexity of arithmetic locally symmetric spaces},
  journal = {Invent. Math.},
  volume  = {244},
  year    = {2026},
  pages   = {143--164},
  doi     = {10.1007/s00222-025-01394-1},
  url     = {https://doi.org/10.1007/s00222-025-01394-1}
}

@article{CS,
  author  = {Culler, Marc and Shalen, Peter B.},
  title   = {The volume of a hyperbolic 3-manifold with {B}etti number 2},
  journal = {Proceedings of the American Mathematical Society},
  volume  = {120},
  number  = {4},
  pages   = {1281--1288},
  year    = {1994},
  doi     = {10.2307/2160250},
}

@article {Takeuchi,
    AUTHOR = {Takeuchi, Kisao},
     TITLE = {Arithmetic {F}uchsian groups with signature {$(1;e)$}},
   JOURNAL = {J. Math. Soc. Japan},
  FJOURNAL = {Journal of the Mathematical Society of Japan},
    VOLUME = {35},
      YEAR = {1983},
    NUMBER = {3},
     PAGES = {381--407},
      ISSN = {0025-5645,1881-1167},
   MRCLASS = {10D07 (20H10)},
  MRNUMBER = {702765},
MRREVIEWER = {K.-B.\ Gundlach},
       DOI = {10.2969/jmsj/03530381},
       URL = {https://doi.org/10.2969/jmsj/03530381},
}

@book {Maskit,
    AUTHOR = {Maskit, Bernard},
     TITLE = {Kleinian groups},
    SERIES = {Grundlehren der mathematischen Wissenschaften [Fundamental
              Principles of Mathematical Sciences]},
    VOLUME = {287},
 PUBLISHER = {Springer-Verlag, Berlin},
      YEAR = {1988},
     PAGES = {xiv+326},
      ISBN = {3-540-17746-9},
   MRCLASS = {30F40 (20H10 22E40)},
  MRNUMBER = {959135},
MRREVIEWER = {William\ Abikoff},
}
\end{document}